\documentclass[12pt]{amsart}

\usepackage{amsmath}
\usepackage{amsfonts}
\usepackage{cite}

\theoremstyle{plain}
\newtheorem{theorem}{Theorem}[section]
\newtheorem{lemma}[theorem]{Lemma}
\newtheorem{proposition}[theorem]{Proposition}
\newtheorem{corollary}[theorem]{Corollary}
\theoremstyle{definition}
\newtheorem{definition}[theorem]{Definition}
\theoremstyle{remark}
\newtheorem{remark}[theorem]{Remark}

\newcommand{\tworow}[3]{\left(\begin{array}{#1} #2 \\ #3\end{array}\right)}
\newcommand{\tworowsmall}[2]{\genfrac{(}{)}{0pt}{}{#1}{#2}}

\title{Mixed-norm estimates and symmetric geometric means}
\author{Wayne Grey}

\begin{document}

\begin{abstract}
The mixed-norm versions of the H\"older and Minkowski integral inequalities are used to produce
new, general estimates involving symmetric geometric means of mixed norms. Various existing
mixed-norm estimates are shown to be simple special cases of these new results.
Examples are also given of applying mixed-norm H\"older and Minkowski to other estimates,
providing much easier proofs.
Finally, the effectiveness of this technique is demonstrated by deriving a new inequality which combines
features from two separate previous results.
\end{abstract}

\maketitle

\section{Introduction}

Although mixed-norm $L^P$ spaces were described by Benedek and Panzone \cite{benedek_panzone}
in 1961, their applications have appeared in the literature at least since Littlewood's 4/3 inequality
\cite{Littlewood1930} in 1930, a fundamental step in bilinearity and a precursor
to Grothendieck's later multilinearity work \cite{Grothendieck-resume}.
This inequality is generalized by the Bohnenblust-Hille inequality, for which
recent advances \cite{BH-hypercontractive} have been achieved through techniques including
mixed-norm estimates.

Fournier \cite{Fournier1987} developed a mixed-norm approach to Sobolev embeddings,
followed by the work of authors including Algervik and Kolyada \cite{AlgervikKolyada2011},
as well as Clavero and Soria \cite{ClaveroSoria2014}. The notion of symmetric mixed-norm
spaces is central to this work, so much so that in \cite{ClaveroSoria2014} they are simply
called ``mixed norm spaces". That paper uses ``Benedek-Panzone spaces" to refer to those
spaces which are called mixed-norm spaces in \cite{benedek_panzone} and here.
Estimates by geometric means of mixed norms, similarly symmetric in the sense that each
mixed norm involved features the same exponents but differently permuted variables,
appear frequently in the literature; see \cite{blei:fractional_cartesian_products},
\cite{BH-hypercontractive}, \cite{PopaSinnamon2013}, and even \cite{Littlewood1930}.

Such estimates are useful, but have often been established by tricky inductions
on the number of variables, using the classical (one-variable) H\"older's inequality and
Minkowski integral inequality.
The difficulty of these proofs not only hinders communication, but makes it harder to
find strong results.
The mixed-norm generalizations of the H\"older inequality \cite{benedek_panzone}
and the Minkowski integral inequality \cite{Fournier1987} can be used to
simplify many arguments, but are too often overlooked.

Section \ref{sec:Minkowski} develops the Minkowski integral inequality for mixed norms.
Although this theorem is known, this description gives a more general statement and perhaps
explains more detail than other available treatments, as well as using notation more
suited to the main results to follow.
(Another description, with different notation, is in the thesis \cite{grey:thesis},
where the appendix gives some of the applications here.)
Section \ref{sec:main} provides the main new results, Theorem \ref{thm:Holder_symmetric}
and Corollary \ref{cor:HM_symmetric}, estimates where the upper bounds are
symmetric geometric means of mixed norms. These give general embeddings of
symmetric mixed-norm spaces into Lebesgue spaces, requiring no more
computation than finding harmonic means.

Section \ref{sec:applications} shows that various known estimates are simple special
cases of these results. Section 5 treats examples where these theorems do not apply,
but mixed-norm H\"older and Minkowski still simplify the proofs.
Finally, Theorem \ref{thm:new_blei_ps} is a new result which combines
features of existing estimates in a more complicated inequality, which is nonetheless
fairly straightforward to establish with mixed-norm techniques.

In some cases, stronger embedding results have been proven than those given here.
For example, Fournier's \cite{Fournier1987} and, together with Blei, \cite{BleiFournier1989}
give embeddings into Lorentz spaces $\ell^{r,1}$, stronger than the embeddings into $\ell^r$
which would be obtained with the methods given here. Milman \cite{Milman1991} uses
interpolation to produce similar embeddings.
Algervik and Kolyada \cite{AlgervikKolyada2011} establish embeddings of symmetric
mixed-norm spaces into Lorentz spaces, and Clavero and Soria \cite{ClaveroSoria2014}
extend this work to more general rearrangement-invariant spaces.
But, while powerful, these results tend to be somewhat restricted, requiring that the
mixed norms be of a particular form or feature certain exponents.
In contrast, the results here apply quite generally, and may be hoped to lead to
stronger future results for Lorentz or other spaces.

\section{Mixed-norm Minkowski integral inequality}
\label{sec:Minkowski}

While the Minkowski integral inequality is fundamentally a mixed-norm inequality in
two variables, it has a natural generalization to mixed norms in more variables.
Fournier developed a mixed-norm Minkowski in \cite{Fournier1987}, giving the
key ideas but stating the theorem for fully-sorted mixed norms. That version is given
here as Corollary \ref{cor:Minkowski_sorted}. This paper also coined the term
``raises" to describe transpositions; this property is given for more
general permutations in Definition \ref{defn:perm_raise_lower}.

\begin{definition}
Let $\left(X_1, \mu_1\right), \ldots, \left(X_n, \mu_n\right)$ be $\sigma$-finite measure spaces,
with the product space $\left(X, \mu\right)$.
For any $p_1, \ldots, p_n \in \left(0,\infty\right]$, we can define a mixed norm of a measurable
function $f(x_1, \ldots, x_n) : X \to \mathbb{C}$ by first specifying
a double $n$-tuple
\begin{equation*}
P = \tworow{cccc}{p_1 & p_2 & \cdots & p_n}{x_1 & x_2 & \cdots & x_n},
\end{equation*}
in terms of which the mixed norm is
\begin{equation*}
\left\| f \right\|_P
= \left( \int_{X_n} \cdots
	\left( \int_{X_1} \left| f(x_1, \ldots, x_n) \right|^{p_1}
	d\mu_1(x_1) \right)^{p_2/p_1}
\cdots d\mu_n(x_n) \right)^{1/p_n},
\end{equation*}
as long as each $p_j < \infty$ (for $j \in \left\{1, \ldots, n\right\}$).
As in classical $L^p$, if any $p_j = \infty$, replace by the essential supremum in that variable.
\end{definition}

\begin{remark}
$\left\|\cdot\right\|_P$ is only a norm when every $p_j \geq 1$; otherwise, the
triangle inequality fails. Unless otherwise specified, however, ``mixed norm"
will be used here to include any $\left\|\cdot\right\|_P$, even if not strictly
speaking a norm.
\end{remark}

Because the value of $\left\| f \right\|_P$ depends only on the
modulus $\left|f\right|$, we need only consider $f \geq 0$.

\begin{definition}
Let $L^+(X)$ denote
the cone of nonnegative measurable functions on $X$.
\end{definition}

\begin{definition}
\label{def:Minkowski_action}
If $\sigma$ is a permutation of $\left\{1, \ldots, n\right\}$ and
$P = \tworowsmall{p_1 \,\cdots\, p_n}{x_1 \,\cdots\, x_n}$, then
\begin{equation*}
P \cdot \sigma
= \tworow{ccccc}{p_{\sigma(1)} & \cdots & p_{\sigma(j)} & \cdots & p_{\sigma(n)}}
	{x_{\sigma(1)} & \cdots & x_{\sigma(j)} & \cdots & x_{\sigma(n)}}.
\end{equation*}
Extend this to $P$ where the variables are not in numeric order
by relabeling the variables.
\end{definition}

\begin{remark}
This defines a right group action of
the symmetric group $S_n$, as for any $\sigma, \rho \in S_n$,
\begin{equation*}
(P \cdot \sigma) \cdot \rho = P \cdot (\sigma \rho).
\end{equation*}
\end{remark}

\begin{lemma}
\label{lem:Minkowski_swap}
Suppose that $p_1, \ldots, p_n \in \left(0, \infty\right]$,
\begin{equation*}
P = \tworow{cccccc}{p_1 & \cdots & p_j & p_{j+1} & \cdots & p_n}
	{x_1 & \cdots & x_j & x_{j+1} & \cdots & x_n},
\end{equation*}
$1 \leq j < n$, and $p_j \leq p_{j+1}$. Let $\tau$
denote the transposition which swaps $j$ and $j+1$, fixing all
other values in $\left\{1, \ldots, n\right\}$. Then, for any
$f(x_1, \ldots, x_n) \in L^+(X)$,
\begin{equation*}
\left\| f \right\|_P \leq \left\| f \right\|_{P \cdot \tau}.
\end{equation*}
\end{lemma}

\begin{proof}
Define the function
\begin{equation*}
g(x_j, \ldots, x_n)
= \left( \int_{X_{j-1}} \cdots \left( \int_{X_1} f^{p_1} d\mu_1(x_1) \right)^{p_2/p_1} \cdots d\mu_{j-1}(x_{j-1}) \right)^{1/p_{j-1}},
\end{equation*}
which computes a mixed norm over the first $j-1$ variables (if $j=1$, these are zero variables, so this is interpreted as $g=f$),
depending on the remaining variables. Fixing $x_{j+2}, \ldots, x_n$ (i.e. every variable after $x_{j+1}$), the
Minkowski integral inequality, applied with the exponent $\frac{p_{j+1}}{p_j} \geq 1$, shows that 
\begin{align*}
\left\| g \right\|_{\tworowsmall{p_j \, p_{j+1}}{x_j \, x_{j+1}}}
&= \left( \int_{X_{j+1}} \left( \int_{X_j} g^{p_j} d\mu_{p_j} \right)^\frac{p_{j+1}}{p_j} d\mu_{p_{j+1}} \right)^\frac{1}{p_{j+1}} \\
&\leq \left( \int_{X_j} \left( \int_{X_{j+1}} g^{p_{j+1}} d\mu_{p_{j+1}} \right)^\frac{p_j}{p_{j+1}} d\mu_{p_j} \right)^\frac{1}{p_j} \\
&\leq \left\| g \right\|_{\tworowsmall{p_{j+1} \, p_j}{x_{j+1} \, x_j}}.
\end{align*}
This can be interpreted as an inequality of functions of $x_{j+2}, \ldots, x_n$.
Both the integral and essential supremum are order-preserving on nonnegative functions.
Consequently, if $0 \leq f_1 \leq f_2$, then for any $L^p$ or mixed norm $\left\|\cdot\right\|$,
$\left\| f_1 \right\| \leq \left\| f_2 \right\|.$

Therefore we can apply the mixed norm
$\tworowsmall{p_{j+2} \, \cdots \, p_n}{x_{j+2} \, \cdots \, x_n}$
in the remaining variables to both sides
above, yielding
\begin{equation*}
\left\| f \right\|_P
= \left\| \left\| g \right\|_{\tworowsmall{p_j \, p_{j+1}}{x_j \, x_{j+1}}}
	\right\|_{\tworowsmall{p_{j+2} \,\cdots\, p_n}{x_{j+2} \,\cdots\, x_n}}
\leq \left\| \left\| g \right\|_{\tworowsmall{p_{j+1} \, p_j}{x_{j+1} \, x_j}}
	\right\|_{\tworowsmall{p_{j+2} \,\cdots\, p_n}{x_{j+2} \,\cdots\, x_n}}
= \left\| f \right\|_{P \cdot \tau}.
\end{equation*}
\end{proof}

\begin{definition}
\label{defn:perm_raise_lower}
With
\begin{equation*}
P = \tworow{ccc}{p_1 & \ldots & p_n}{x_1 & \ldots & x_n}
\end{equation*}
a permutation $\sigma$ \textit{raises} $P$
if $p_i~\leq~p_j$ whenever $i~<~j$ and $\sigma^{-1}(j)~<~\sigma^{-1}(i)$.
Similarly, a permutation $\sigma$ \textit{lowers} $P$
if $p_j~\leq~p_i$ whenever $i~<~j$ and $\sigma^{-1}(j)~<~\sigma^{-1}(i)$.
\end{definition}

\begin{remark}
\label{rem:trans_raise_lower}
An adjacent transposition $\tau = (\begin{array}{cc}j & j+1\end{array})$ raises
\begin{equation*}
P \cdot \sigma = \tworow{ccc}{p_{\sigma(1)} & \cdots & p_{\sigma(n)}}{x_{\sigma(1)} & \cdots & x_{\sigma(n)}}
\end{equation*}
if and only if $p_{\sigma(j)} \leq p_{\sigma(j+1)}$. Similarly, this $\tau$ lowers
$P \cdot \sigma$ if and only if $p_{\sigma(j+1)} \leq p_{\sigma(j)}$.
\end{remark}

\begin{lemma}
\label{lem:inverse_raise_lower}
A permutation $\sigma$ raises $P$ if and only if $\sigma^{-1}$ lowers $P \cdot \sigma$.
(Equivalently, $\sigma$ lowers $P$ if and only if $\sigma^{-1}$ raises $P \cdot \sigma$.)
\end{lemma}

\begin{proof}
As defined, $\sigma$ raises $P$ if and only if $p_i \leq p_j$ whenever $i < j$ and
$\sigma^{-1}(j) < \sigma^{-1}(i)$. Let $b = \sigma^{-1}(i)$ and $a = \sigma^{-1}(j)$,
and observe that this is equivalent to saying that $p_{\sigma(b)} \leq p_{\sigma(a)}$
whenever $a < b$ and $\sigma(b) < \sigma(a)$, i.e. that $\sigma^{-1}$ lowers
$P \cdot \sigma$.

To see that the second formulation is equivalent, just swap $\sigma$ and $\sigma^{-1}$,
$P$ and $P \cdot \sigma$, and note that $P \cdot \sigma \cdot \sigma^{-1} = P$.
\end{proof}

\begin{lemma}
\label{lem:composition_raise_lower}
If $\sigma$ raises $P$ and $\rho$ raises $P \cdot \sigma$, then $\sigma\rho$ raises $P$.
Similarly, if $\sigma$ lowers $P$ and $\rho$ lowers $P \cdot \sigma$, then
$\sigma\rho$ lowers $P$.
\end{lemma}

\begin{proof}
Suppose that $\sigma$ raises $P$ and that $\rho$ raises $P \cdot \sigma$.
Consider any $i < j$ such that $(\sigma\rho)^{-1}(j) < (\sigma\rho)^{-1}(i)$.

If $\sigma^{-1}(j) < \sigma^{-1}(i)$, then $p_i \leq p_j$, because $\sigma$ raises $P$
and $i < j$. Otherwise, $\sigma^{-1}(i) < \sigma^{-1}(j)$ and
$\rho^{-1}(\sigma^{-1}(j)) < \rho^{-1}(\sigma^{-1}(i))$.
Because $\rho$ raises $P \cdot \sigma$, this means that
$\left(P \cdot \sigma\right)_{\sigma^{-1}(i)} \leq \left(P \cdot \sigma\right)_{\sigma^{-1}(j)}$,
i.e. $p_i = p_{\sigma(\sigma^{-1}(i))} \leq p_{\sigma(\sigma^{-1}(j))} = p_j$.

Either way, $p_i \leq p_j$, so $\sigma\rho$ raises $P$.

Next, assume that $\sigma$ lowers $P$ and $\rho$ lowers $P \cdot \sigma$. By
Lemma \ref{lem:inverse_raise_lower}, this means that $\rho^{-1}$ raises
$(P \cdot \sigma) \cdot \rho = P \cdot \sigma\rho$, and $\sigma^{-1}$ raises
$P \cdot \sigma$.
By the previous part of this lemma, $\rho^{-1}\sigma^{-1}$ raises $P\cdot\sigma\rho$.
Applying Lemma \ref{lem:inverse_raise_lower} again, this means that
$\sigma\rho$ lowers $P$, as desired.
\end{proof}

\begin{theorem}
\label{thm:adjacent_swaps}
Any permutation raises $P$ if and only if it is a composition $\tau_1 \cdots \tau_m$
(for some $m \geq 0$)
of adjacent transpositions such that, for each $1 \leq k \leq m$, $\tau_k$ raises
$P \cdot \tau_1 \cdots \tau_{k-1}$.

Similarly, any permutation lowers $P$ if and only if it is a composition of adjacent
transpositions $\tau_1 \cdots \tau_m$ such that each $\tau_k$ lowers
$P \cdot \tau_1 \cdots \tau_{k-1}$.
\end{theorem}

\begin{proof}
If $\sigma = \tau_m \cdots \tau_1$ is a composition as specified, each
$\tau_k$ raising (or lowering) $P \cdot (\tau_{k-1} \cdots \tau_1)$, then
$\sigma$ raises (or lowers) $P$, by Lemma \ref{lem:composition_raise_lower}.

Now suppose that $\sigma$ raises $P$. The proof that it is a composition of
adjacent transpositions as above is by induction on the number of inversions
in $\sigma$, i.e. the number of pairs $i < j$ such that $\sigma(j) < \sigma(i)$.
As a base case, the identity is an empty composition. It is impossible to have
$\sigma(1) \leq \cdots \leq \sigma(n)$ unless $\sigma$ is the identity, so
any non-identity $\sigma$ must have at least one inverted adjacent pair,
say $\sigma(k+1) < \sigma(k)$.

Let $a = \sigma(k+1)$ and $b = \sigma(k)$ and note that $a < b$ and
$\sigma^{-1}(b) < \sigma^{-1}(a)$, so because $\sigma$ raises $P$,
$p_a \leq p_b$. Let $\tau = \left(\begin{array}{cc} k & k+1\end{array}\right)$
and observe that
\begin{equation*}
P \cdot \sigma\tau =
\tworow{cccccccc}
{p_{\sigma(1)} & \cdots & p_{\sigma(k-1)} & p_a & p_b & p_{\sigma(k+2)} & \cdots & p_{\sigma(n)}}
{x_{\sigma(1)} & \cdots & x_{\sigma(k-1)} & x_a & x_b & x_{\sigma(k+2)} & \cdots & x_{\sigma(n)}}.
\end{equation*}
For any pair $i < j$, $\sigma(i)$ and $\sigma(j)$ are in the same relative order as $\sigma\tau(i)$
and $\sigma\tau(j)$ unless the pair consists of $k$ and $k+1$. Because $\sigma\tau(k) = \sigma(k+1)
< \sigma(k) = \sigma\tau(k+1)$ and $\sigma$ raises $P$, $\sigma\tau$ also raises $P$.
Since $\sigma\tau$ has one fewer inversion than $\sigma$ (as $\sigma\tau(k) = a < b = \sigma\tau(k+1)$),
by the inductive hypothesis, there are adjacent transpositions $\tau_1, \ldots, \tau_m$ such that
$\sigma\tau = \tau_1 \cdots \tau_m$ and each $\tau_k$ raises $P \cdot \tau_1 \cdots \tau_{k-1}$.

Finally, $\tau = \left(\begin{array}{cc} k & k+1\end{array}\right)$ raises $P \cdot \sigma\tau$,
because $a < b$, $\tau^{-1}(b) = k < k+1 = \tau^{-1}(a)$, and $p_a \leq p_b$. Therefore we
let $\tau_{m+1} = \tau$ and have $\sigma = \tau_1 \cdots \tau_{m+1}$ as desired.

Now, if $\sigma$ lowers $P$, then $\sigma^{-1}$ raises $P \cdot \sigma$
by Lemma \ref{lem:inverse_raise_lower}. The preceding characterization shows that
$\sigma^{-1} = \tau_1 \cdots \tau_m$ as a composition of adjacent transpositions, where
each $\tau_k$ raises $P \cdot \sigma \tau_1 \cdots \tau_{k-1} = P \cdot \tau_m \cdots \tau_k$.
Therefore $\sigma = \tau_m \cdots \tau_1$, where by Lemma \ref{lem:inverse_raise_lower}
each $\tau_k = \tau_k^{-1}$ lowers $P \cdot \tau_m \cdots \tau_{k+1}$.
This is the desired result, up to relabeling each $\tau_k$ as $\tau_{m-k}$.
\end{proof}

\begin{theorem}[Mixed-norm Minkowski integral inequality]
\label{thm:Minkowski_mixed}
If $\sigma$ is a permutation which raises $P$, then for any $f \in L^+(X)$,
$\left\| f \right\|_P \leq \left\| f \right\|_{P \cdot \sigma}$.

Similarly, if $\rho$ lowers $P$, then for any $f \in L^+(X)$, $\left\| f \right\|_{P \cdot \rho}
\leq \left\| f \right\|_P$.
\end{theorem}

\begin{proof}
Suppose that $\sigma$ raises $P$. Use Theorem \ref{thm:adjacent_swaps} to write
$\sigma = \tau_1 \cdots \tau_m$, where each $\tau_k$ raises $P \cdot \tau_1 \cdots \tau_{k-1}$.
Between Remark \ref{rem:trans_raise_lower} and Lemma \ref{lem:Minkowski_swap}, for any $f \in L^+(X)$,
\begin{equation*}
\left\| f \right\|_P
\leq \left\| f \right\|_{P \cdot \tau_1}
\leq \cdots
\leq \left\| f \right\|_{P \cdot \tau_1 \cdots \tau_m}
= \left\| f \right\|_{P \cdot \sigma}.
\end{equation*}

The proof when $\rho$ lowers $P$ is similar, with the inequalities reversed.
\end{proof}

\begin{corollary}[Fournier's fully-sorted Minkowski]
\label{cor:Minkowski_sorted}
Let
\begin{equation*}
P = \tworow{ccc}{p_1 & \cdots & p_n}{x_1 & \cdots & x_n},
\end{equation*}
and let $\sigma, \rho \in S_n$ be permutations such that
\begin{equation*}
p_{\sigma(1)} \geq p_{\sigma(2)} \geq \cdots \geq p_{\sigma(n)}
\quad \text{ and } \quad
p_{\rho(1)} \leq p_{\rho(2)} \leq \cdots \leq p_{\rho(n)}.
\end{equation*}
Then, for any $f \in L^+(X)$,
\begin{equation*}
\left\| f \right\|_{P \cdot \rho}
\leq \left\| f \right\|_P
\leq \left\| f \right\|_{P \cdot \sigma}.
\end{equation*}
\end{corollary}

\begin{proof}
Any list can be sorted by adjacent swaps of out-of-order elements; see, for example, the bubble
sort algorithm, as described in \cite[pp.~106--111]{knuth-sorting}.
Such sorting of the exponents into numeric order takes $P$ to $P \cdot \rho$, for some
$\rho \in S_n$ which lowers $P$, as defined in Definition \ref{defn:perm_raise_lower}.
Sorting into reverse numeric order takes $P$ to some $P \cdot \sigma$, where $\sigma$
raises $P$.

By the mixed-norm Minkowski integral inequality in Theorem \ref{thm:Minkowski_mixed},
\begin{equation*}
\left\| f \right\|_{P \cdot \rho}
\leq \left\| f \right\|_P
\leq \left\| f \right\|_{P \cdot \sigma}.
\end{equation*}
\end{proof}

\section{Estimates with symmetric geometric means of mixed norms}
\label{sec:main}

Again, let $(X_1, \mu_1), \ldots, (X_n, \mu_n)$ be $\sigma$-finite measure spaces with product
$(X, \mu)$. Recall the mixed-norm H\"older inequality given by Benedek and Panzone early in
\cite{benedek_panzone}.
(Note that this theorem can be proven by applying the $m$-function H\"older's inequality
in each variable successively.)

\begin{proposition}[Mixed-norm H\"older inequality]
\label{pro:Holder_mixed}
Let $f_1, \ldots, f_m \in L^+(X)$ be any finitely many functions, with corresponding
double $n$-tuples $P_1, \ldots, P_m$, each
\begin{equation}
\label{Holder_Pi}
P_i = \tworow{ccc}{p_{i,1} & \cdots & p_{i,n}}{x_1 & \cdots & x_n}
\end{equation}
such that $\sum_{i=1}^m P_i^{-1} = 1$, understood coordinatewise. That is, for each
$j \in \left\{1, \ldots, n\right\}$,
$\sum_{i=1}^m p_{i,j}^{-1} = 1.$
Then
\begin{equation*}
\int_X f_1 \cdots f_m d\mu
\leq \left\| f_1 \right\|_{P_1} \cdots \left\| f_m \right\|_{P_m}.
\end{equation*}
\end{proposition}

\begin{definition}
\label{defn:harmonic_mean}
Given
\begin{equation*}
P = \tworow{ccc}{p_1 & \cdots & p_n}{x_1 & \cdots & x_n},
\end{equation*}
denote the harmonic mean of the exponents in $P$ by
\begin{equation*}
\overline{p} = \left( \frac{1}{n} \sum_{j=1}^n p_j^{-1} \right)^{-1}.
\end{equation*}
\end{definition}

\begin{definition}
\label{defn:up_down_action}
Define two more right actions of the symmetric group $S_n$ by,
for any $\sigma \in S_n$, letting
\begin{equation*}
P^\sigma = \tworow{ccc}{p_{\sigma(1)} & \cdots & p_{\sigma(n)}}{x_1 & \cdots & x_n}
\text { and }
P_\sigma = \tworow{ccc}{p_1 & \cdots & p_n}{x_{\sigma(1)} & \cdots & x_{\sigma(n)}}.
\end{equation*}

\end{definition}

\begin{definition}
\label{defn:orbit_size_m}
Let $m$ denote the size of the orbit $\left\{ P^\sigma : \sigma \in S_n \right\}$ of $P$.
If the exponents $\left\{p_1, \ldots, p_n\right\}$ have $r$ many distinct values
$v_1, \ldots, v_r$, such that each value $v_k$ occurs $n_k$ many times, then
\begin{equation*}
m = \frac{n!}{n_1! \cdots n_r!}.
\end{equation*}
\end{definition}

\begin{theorem}
\label{thm:Holder_symmetric}
Given a fixed $P$,
let its orbit $\left\{P^\sigma : \sigma \in S_n\right\}$
be enumerated by $P_1, \ldots, P_m$.
For any functions $f_1, \ldots, f_m \in L^+(X)$, 
\begin{equation*}
\left\| \prod_{i=1}^m f_i^{1/m} \right\|_{L^{\overline{p}}(X)}
\leq \prod_{i=1}^m \left\| f_i \right\|_{P_i}^{1/m}.
\end{equation*}
\end{theorem}

\begin{proof}
This result is trivial if all exponents are the same, with both sides $L^{\overline{p}}(X)$
norms of a single function. Therefore assume this is not the case, implying in particular
that $\overline{p} < \infty$ and that $m \geq n$.

For each $1 \leq i \leq m$, let
\begin{equation*}
Q_i = \tworow{ccc}{m p_{i,1} / \overline{p} & \cdots & m p_{i,n} / \overline{p}}{x_1 & \cdots & x_n},
\end{equation*}
with $P_i$ as in (\ref{Holder_Pi}). Observe that, for each $i$ and
any $1 \leq j \leq n$, $m p_i(j) \geq 1$, because since $m \geq n$,
\begin{equation*}
\frac{m p_{i,j}}{\overline{p}}
= \frac{m}{n} p_{i,j} \sum_{k=1}^n p_{k,j}^{-1}
\geq \frac{m}{n} \left(1 + \sum_{k \neq j} \frac{p_{i,j}}{p_{k,j}}\right) \geq 1.
\end{equation*}
Furthermore, $\sum_{i=1}^m Q_i^{-1} = 1$ coordinatewise.
To see this, fix any $l \in \left\{1, \ldots, n\right\}$ and $k \in \left\{1, \ldots, r\right\}$.
The number of $P^\sigma$ in the orbit of $P$ which place the value $v_k$
(which appears $n_k$ times in the top row of $P$) in the $l^{th}$ position is then
\begin{equation*}
\frac{(n-1)!}{n_1! \cdots n_{k-1}! (n_k-1)! n_{k+1}! \cdots n_r!}
= \frac{n_k}{n} m.
\end{equation*}
Therefore
\begin{equation*}
\sum_{i=1}^m \frac{\overline{p}}{m p_{i,l}}
= \frac{\overline{p}}{m} \sum_{i=1}^m p_{i,l}^{-1}
= \frac{\overline{p}}{n} \sum_{k=1}^r \frac{n_k}{v_k}
= \frac{\overline{p}}{n} \sum_{j=1}^n p_j^{-1}
= 1,
\end{equation*}
by the definition of $\overline{p}$,
so Proposition \ref{pro:Holder_mixed} (H\"older's inequality) can
be applied to the functions $f_1^{\overline{p}/m}, \ldots,
f_m^{\overline{p}/m}$, yielding
\begin{equation*}
\int_X \prod_{i=1}^m f_i^{\overline{p}/m}
\leq \prod_{i=1}^m \| f_i^{\overline{p}/m} \|_{Q_i}
= \prod_{i=1}^m \| f_i \|_{P_i}^{m/\overline{p}}.
\end{equation*}
Take the $\overline{p}/m$ power of each side for the desired result.
\end{proof}

One mixed norm may be defined by several different double $n$-tuples. For example, if
\begin{equation*}
P_1 = \tworow{ccc}{3 & 2 & 2}{x_1 & x_2 & x_3}
\text{ and }
P_2 = \tworow{ccc}{3 & 2 & 2}{x_1 & x_3 & x_2},
\end{equation*}
then for any measurable $f(x_1, x_2, x_3) \geq 0$,
\begin{equation*}
\left\| f \right\|_{P_1}
= \left( \int_{X_2 \times X_3} \left( \int_{X_1} f^3 d\mu_1 \right)^{2/3} d(\mu_2 \times \mu_3) \right)^{1/2}
= \left\| f \right\|_{P_2}
\end{equation*}
by Tonelli's theorem. In general, the order of the variables associated with consecutive repeated exponents
does not change the norm. (In this example,
the order of $x_2$ and $x_3$ is immaterial.) Therefore, we identify any double $n$-tuples which differ only
in the order of variables within such blocks of repeated exponents. With this identification, as long as
$P$ satisfies $p_1 \geq \cdots \geq p_n$, a simple counting argument shows that the orbit
$\left\{P_\sigma : \sigma \in S_n\right\}$ has the same number of elements
$m$ (from Definition \ref{defn:orbit_size_m})
as the orbit $\left\{P^\sigma : \sigma \in S_n\right\}$.

Furthermore, whenever $p_1 \geq \cdots \geq p_n$, $P$ is maximal in its orbit for Theorem
\ref{thm:Minkowski_mixed} (the mixed-norm Minkowski inequality), in the sense that
for each $\sigma \in S_n$,
$\left\| f \right\|_{P \cdot \sigma} \leq \left\| f \right\|_P$ for any $f \in L^+(X)$.
These two properties lead to the following result.
Although it closely resembles Theorem \ref{thm:Holder_symmetric},
from which it is derived, note that here we consider the double $n$-tuples $P_\sigma$
rather than $P^\sigma$. This means that, while Theorem \ref{thm:Holder_symmetric}
permutes the exponents while leaving the order of the variables fixed, here the
exponents keep their order while the variables are permuted.

\begin{corollary}
\label{cor:HM_symmetric}
Given a fixed $P$ with $p_1 \geq \cdots \geq p_n$,
let its orbit $\left\{P_\sigma : \sigma \in S_n\right\}$,
modulo the above identification, be enumerated by $P_1, \ldots, P_m$.
For any functions $f_1, \ldots, f_m \in L^+(X)$,
\begin{equation*}
\left\| \prod_{i=1}^m f_i^{1/m} \right\|_{L^{\overline{p}}(X)}
\leq \prod_{i=1}^m \left\| f_i \right\|_{P_i}^{1/m}.
\end{equation*}
\end{corollary}

\begin{proof}
For each $P_\sigma$ in the orbit $\left\{P_\sigma : \sigma \in S_n\right\}$, there is a
corresponding $P^{\sigma^{-1}} = P_\sigma \cdot \sigma^{-1}$ in the other orbit,
$\left\{P^\sigma : \sigma \in S_n\right\}$.
Let $Q_1, \ldots, Q_m$ be obtained from $P_1, \ldots, P_m$ in this way; that is,
writing each $P_i = P_{\sigma_i}$, the corresponding $Q_i = P^{\sigma_i^{-1}}$.
These $Q_i$ enumerate the collection of $P^{\sigma^{-1}}$, which is in
fact the orbit $\left\{P^\sigma : \sigma \in S_n\right\}$.

By Theorem \ref{thm:Holder_symmetric},
\begin{equation*}
\left\| \prod_{i=1}^m f_i^{1/m} \right\|_{L^{\overline{p}}(X)}
\leq \prod_{i=1}^m \left\| f_i \right\|_{Q_i}^{1/m}.
\end{equation*}
Because each $P_i$ can be obtained from $Q_i$ by sorting its columns so that the exponents
are in decreasing order, by Corollary \ref{cor:Minkowski_sorted}, each
$\left\|f_i\right\|_{Q_i} \leq \left\| f_i\right\|_{P_i}$.
\end{proof}

\begin{corollary}
\label{cor:HM_symmetric_1fn}
Given a fixed $P$ with $p_1 \geq \cdots \geq p_n$,
let its orbit $\left\{ P_\sigma : \sigma \in S_n\right\}$ be enumerated by $P_1, \ldots, P_m$.
For any $f \in L^+(X)$,
\begin{equation*}
\left\| f \right\|_{L^{\overline{p}}(X)} \leq \prod_{i=1}^m \left\| f \right\|_{P_i}^{1/m}.
\end{equation*}
\end{corollary}

\begin{proof}
Simply apply Corollary \ref{cor:HM_symmetric} with each $f_i = f$.
\end{proof}

\begin{remark}
The exponent $\overline{p}$ on the left-hand side of the inequality in each of
Theorem \ref{thm:Holder_symmetric}, Corollary \ref{cor:HM_symmetric}, and
Corollary \ref{cor:HM_symmetric_1fn} is the only exponent $p$ such that the result is
valid for all $\sigma$-finite measure spaces, even allowing a constant $C$
(depending on the spaces, but not the functions $f_i$) such that
\begin{equation}
\label{eqn:only_exponent}
\left\| \prod_{i=1}^m f_i^{1/m} \right\|_{L^p(X)}
\leq C \prod_{i=1}^m \left\| f_i \right\|_{P_i}^{1/m}.
\end{equation}
(Consider $X_1 = \cdots = X_n = \mathbb{R}$ and each $f_1 = \cdots = f_m =
\prod_{j=1}^n \chi_{[0,t]}(x_j)$, then take limits $t \to 0$ and $t \to \infty$.
Similar examples are possible in any spaces featuring sets of arbitrarily small and
arbitrarily large measure.)
\end{remark}

As an additional note, when using either Corollary \ref{cor:HM_symmetric} or Corollary \ref{cor:HM_symmetric_1fn},
it suffices to specify only the top row as an $n$-tuple $\left(p_1, \ldots, p_n\right)$ with
$p_1 \geq \cdots \geq p_n$, for this is enough to specify both the orbit $\left\{P_\sigma :
\sigma \in S_n\right\}$ and $\overline{p}$.

\section{Applications of main results}
\label{sec:applications}

These results provide an easy way to generate mixed-norm estimates, where most of the
computational work is finding the harmonic mean $\overline{p}$. Many estimates in
the literature are simple consequences of Theorem \ref{thm:Holder_symmetric} and
Corollary \ref{cor:HM_symmetric}, and can now be easily proven and generalized.

Perhaps the simplest application is a mixed-norm intermediate result to Littlewood's
4/3 inequality, a fundamental step in the theory of multilinearity, and an early example
of the importance of $L^p$ for exponents $p$ other than the ubiquitous $1$, $2$, and
$\infty$. One modern source describing Littlewood's 4/3 inequality is
Garling's book \cite{garling-inequalities}, where the proof of the inequality, there
Corollary 18.1.1, establishes and uses this mixed-norm estimate.

As with many of these sorts of results, the original was given for sums, but these
methods easily generalize it to integrals.

\begin{proposition}
\label{pro:4/3}
For any $\sigma$-finite measure spaces $(X, \mu)$ and $(Y, \nu)$ and any function
$f(x,y) \in L^+(X \times Y)$,
\begin{align*}
\left( \int_{X \times Y} f^{\frac{4}{3}} d\mu d\nu \right)^{\frac{3}{4}}
&= \left( \int_Y \left( \int_X f^2 d\mu \right)^{\frac{1}{2}} d\nu \right)^{\frac{1}{2}}
\left( \int_X \left( \int_Y f^2 d\nu \right)^{\frac{1}{2}} d\mu \right)^{\frac{1}{2}}.
\end{align*}
\end{proposition}

\begin{proof}
Use Corollary \ref{cor:HM_symmetric_1fn} with
$P = \tworowsmall{2 \, 1}{x \, y}$,
so $\overline{p} = \left( \frac{2^{-1} + 1^{-1}}{2} \right)^{-1} = \frac{4}{3}$.
\end{proof}

Blei gives a similar 6/5 inequality with three variables in Lemma 2 on page 430 of
\cite{blei:int_frac_dimension}, again stated for series but easily generalized to integrals
on any $\sigma$-finite spaces. To produce and prove this result, simply apply Corollary
\ref{cor:HM_symmetric_1fn} with $P = \left(2, 1, 1\right)$, so $\overline{p} = 6/5$.

These results find a generalization in Blei's Lemma 5.3 from \cite{blei:fractional_cartesian_products},
which considers exponents $2$ and $1$, each appearing arbitrarily often.
A special case of this mixed-norm estimate was used as Lemma 1 in \cite{BH-hypercontractive},
a paper using multilinear techniques to study the Bohnenblust-Hille inequality.
Preliminary definitions are followed by a generalization of Blei's result from sums to integrals.

\begin{definition}
Consider integers $J > K > 0$. Let $N = \binom{J}{K}$ and let $S_1, \ldots, S_N$ enumerate the subsets of
$\left\{1, \ldots, J\right\}$ with cardinality $K$. For $1~\leq~\alpha~\leq~N$, let $\sim S_\alpha$ denote the
complement $\left\{1, \ldots, J\right\} \setminus S_\alpha$.
\end{definition}

\begin{proposition}
\label{pro:Blei_21}
For any $\sigma$-finite measure spaces $(X_1, \mu_1), \ldots, (X_J, \mu_j)$ and any measurable function
$f(x_1, \ldots, x_J)$ on $X_1 \times \cdots \times X_J$,
\begin{equation*}
\left( \int_{\left\{1, \ldots, J\right\}} \left|f\right|^\frac{2J}{K+J} \right)^\frac{K+J}{2J}
\leq \prod_{\alpha=1}^N \left[ \int_{S_\alpha} \left( \int_{\sim S_\alpha} |f|^2 \right)^{1/2} \right]^{1/N},
\end{equation*}
where for any subset $E \subset \left\{1, \ldots, J\right\}$, the notation $\int_E$ denotes integration over the product
space $\prod_{k \in E} X_k$.
\end{proposition}

\begin{proof}
To prepare for Corollary \ref{cor:HM_symmetric_1fn}, 
let $P = \left(\begin{array}{cccccc} 2 & \cdots & 2 & 1 & \cdots & 1 \end{array}\right)$, with $K$ copies of $1$ and
$J-K$ copies of $2$. There are exactly $\binom{J}{K}$ norms in the orbit of $P$, because each such norm is determined
by choosing $K$ variables to place with the $1$ exponents. The $K$ indices of these variables form a subset $S_\alpha$
of $\left\{1, \ldots, J\right\}$.
With the remaining variables, in $\sim S_\alpha$, associated
with the exponent $2$, we form a mixed norm $P_\alpha$ such that
\begin{equation*}
\left\| f \right\|_{P_\alpha} = \int_{S_\alpha} \left( \int_{\sim S_\alpha} |f|^2 \right)^{1/2}.
\end{equation*}
With $K$ copies of $1$ and $J-K$ copies of $2$, the harmonic mean is
\begin{equation*}
\overline{p} = \left( \frac{ K + \frac{1}{2}(J-K)}{J} \right)^{-1} = \frac{2J}{K+J},
\end{equation*}
so the desired result follows from Corollary \ref{cor:HM_symmetric_1fn}.
\end{proof}

Blei's method of proof rests on the same H\"older and Minkowski foundations, but takes three pages in an
induction over single-variable H\"older rather than using mixed-norm techniques. Not only do we have a
quicker and easier proof, but it is now straightforward to find generalizations beyond the exponents $1$
and $2$.

\begin{proposition}
\label{pro:Blei_qp}
Suppose
that
$0 < p < q \leq \infty$.
For any $\sigma$-finite measure spaces $(X_1, \mu_1), \ldots, (X_J, \mu_j)$ and any measurable function
$f(x_1, \ldots, x_J)$ on $X_1 \times \cdots \times X_J$,
\begin{equation*}
\left\| f \right\|_{\frac{Jpq}{pJ + (q-p)K}}
\leq \prod_{\alpha=1}^N \left[ \int_{S_\alpha} \left( \int_{\sim S_\alpha} |f|^q \right)^{p/q} \right]^{1/Np},
\end{equation*}
where for any subset $E \subset \left\{1, \ldots, J\right\}$, the notation $\int_E$ denotes integration over the product
space $\prod_{k \in E} X_k$.
\end{proposition}

\begin{proof}
Let $P = \left(\begin{array}{cccccc} q & \cdots & q & p & \cdots & p \end{array}\right)$, with $K$ copies of $p$
and $J-K$ copies of $q$.
The harmonic mean is
\begin{equation*}
\overline{p} = \left( \frac{p^{-1}K + q^{-1}(J-K)}{J}\right)^{-1} = \frac{Jpq}{Jp + K(q-p)},
\end{equation*}
and the argument proceeds as in Proposition \ref{pro:Blei_21}.
\end{proof}

This technique could easily produce similar results using three or more distinct exponents,
but Corollary \ref{cor:HM_symmetric_1fn} already addresses arbitrarily many.

\begin{remark}
Each of Propositions \ref{pro:4/3}, \ref{pro:Blei_21}, and \ref{pro:Blei_qp} can be easily generalized
to use several functions rather than one, simply by applying Corollary \ref{cor:HM_symmetric} rather
than Corollary \ref{cor:HM_symmetric_1fn}.
\end{remark}

\section{Other mixed-norm estimates}

Although Theorem \ref{thm:Holder_symmetric} and Corollary \ref{cor:HM_symmetric} offer rather polished
results, not every situation calls for these estimates. However, the mixed-norm H\"older and Minkowski inequalities
can be used in other ways, perhaps combined with different techniques.
For example, Theorem \ref{thm:Holder_symmetric} or Corollary \ref{cor:HM_symmetric} do not apply
to Theorems 2.1 and 2.2 in \cite{PopaSinnamon2013}, but the inductive proofs given can be replaced with
much simpler mixed-norm methods. The result follows after suitable definitions.

\begin{definition}
For $j=1, 2, \ldots, n$,
let $(M_j, \mu_j)$ be $\sigma$-finite measure spaces
and define the product measure spaces $(M^n, \mu^n)$
and $(M^n_j, \mu^n_j)$ by
\begin{equation*}
M^n = \prod_{k=1}^n M_k,
\qquad \mu^n = \prod_{k=1}^n \mu_k,
\qquad M^n_j = \underset{k \neq j}{\prod_{k=1}^n} M_k,
\qquad \mu^n_j = \underset{k \neq j}{\prod_{k=1}^n} \mu_k.
\end{equation*}
\end{definition}

\begin{proposition}[Theorems 2.1 and 2.2 in \cite{PopaSinnamon2013}]
\label{pro:PopaSinnamon}
If $n \geq 2$ and $q_1, \ldots, q_n$ are positive (possibly infinite) exponents such that
$\sum_{j=1}^n \frac{1}{q_j} \leq 1$, then
for any nonnegative $\mu^n$-measurable functions $f_1, \ldots, f_n$,
\begin{align*}
\int_{M^n} f_1 \cdots f_n d\mu^n
&\leq \prod_{j=1}^n \left( \int_{M_j} \left( \int_{M^n_j} f_j^{q_j} d\mu^n_j \right)^{p_j/q_j} d\mu_j \right)^{1/p_j} \\
\text{and } \quad
\int_{M^n} f_1 \cdots f_n d\mu^n
&\leq \prod_{j=1}^n \left( \int_{M^n_j} \left( \int_{M_j} f_j^{q_j} d\mu_j \right)^{s_j/q_j} d\mu^n_j \right)^{1/s_j},
\end{align*}
where $\frac{1}{p_j} = \frac{1}{q_j} + 1 - \sum_{k=1}^n \frac{1}{q_k}$
and $\frac{1}{s_j} = \frac{1}{q_j} + \frac{1}{n-1} (1 - \sum_{k=1}^n \frac{1}{q_k})$.
\end{proposition}

\begin{proof}
To prove the first inequality, define, for each $1 \leq j \leq n$,
\begin{equation*}
P_j = \tworow{ccc}{p_{j,1} & \cdots & p_{j,n}}{x_1 & \cdots & x_n},
\end{equation*}
where each $p_{j,j} = p_j$ and, for $j \neq k$, $p_{j,k} = q_j$. The hypotheses ensure that every
$p_{j,k} \geq 1$ and that $\sum_{j=1}^n P_j^{-1} = 1$ coordinatewise, i.e. for each $1 \leq k \leq n$,
$\sum_{j=1}^n \frac{1}{p_{j,k}} = 1$. Therefore H\"older's inequality (Proposition \ref{pro:Holder_mixed})
gives
\begin{equation*}
\int_{M^n} f_1 \cdots f_n d\mu^n
\leq \prod_{j=1}^n \left\| f \right\|_{P_j}.
\end{equation*}
Because each $p_j \leq q_j$, Minkowski's inequality (Corollary \ref{cor:Minkowski_sorted}) gives the
first inequality, where each $L^{p_j}_{\mu_j}$ norm over $X_j$ comes last.

For the second inequality, let
\begin{equation*}
S_j = \tworow{ccc}{s_{j,1} & \cdots & s_{j,n}}{x_1 & \cdots & x_n},
\end{equation*}
where each $s_{j,j} = q_j$ and, for $j \neq k$, $s_{j,k} = s_j$. Again, each $s_{j,k} \geq 1$ and
$\sum_{j=1}^n S_j^{-1}$ coordinatewise. By H\"older,
\begin{equation*}
\int_{M^n} f_1 \cdots f_n d\mu^n
\leq \prod_{j=1}^n \left\| f \right\|_{S_j}.
\end{equation*}
Because each $s_j \leq q_j$, the second inequality follows by Minkowski.
\end{proof}

Not only can many known inequalities be proven easily using mixed-norm techniques, but
generalizations are often in reach, as well. For example, the following inequality combines
the coefficients $q_i$ which do not quite satisfy H\"older (with the gap filled by $p_i$) of
Proposition \ref{pro:PopaSinnamon} (drawn from Popa and Sinnamon \cite{PopaSinnamon2013})
with the variable-sized subsets present in Proposition \ref{pro:Blei_21} based on Blei
\cite{blei:fractional_cartesian_products}.

We resume our initial notation, where $(X_1, \mu_1) \ldots, (X_n, \mu_n)$ are
any $\sigma$-finite measure spaces with product $(X,\mu)$.

\begin{theorem}
\label{thm:new_blei_ps}
Let $0 < k < n$ and $M = \binom{n}{k}$,
and let $S_1, \ldots, S_M$ enumerate the size-$k$ subsets of $\left\{1, \ldots, n\right\}$.
Consider any positive (possibly infinite) exponents $q_1, \ldots, q_M$ such that
$\sum_{i=1}^M \frac{1}{q_i} \leq 1$, and define
$\epsilon = 1 - \sum_{i=1}^M \frac{1}{q_i} \geq 0$.
For any nonnegative numbers $c_1, \ldots, c_M$ such that, for each $j \in \left\{1, \ldots, n\right\}$,
$\sum_{S_i \ni j} c_i = 1$,
and any nonnegative $\mu$-measurable functions $f_1, \ldots, f_M$,
\begin{equation*}
\int_{X} f_1 \cdots f_M d\mu
\leq \prod_{i=1}^M \left( \int_{S_i} \left( \int_{\sim S_i} f_i^{q_i} \right)^{p_i/q_i} \right)^{1/p_i},
\end{equation*}
where
$\frac{1}{p_i} = \frac{1}{q_i} + c_i \epsilon$
and $\int_E$, for $E \subset \left\{1, \ldots, n\right\}$, denotes integration over those $X_j$ with $j \in E$.
\end{theorem}

\begin{remark}
One possible choice of $c_1, \ldots, c_M$ is $c_1 = \cdots = c_M = 1/\binom{n-1}{k-1}$.
When $q_1 = \cdots = q_M$ as well, this leads to Proposition \ref{pro:Blei_qp}.
(In the typical case $N < M$, there are many other choices, as then the system
$\sum_{S_i \ni j} c_i = 1$ is underdetermined.) One can instead let $k$ be either $1$
or $n-1$ to obtain Proposition \ref{pro:PopaSinnamon}.
\end{remark}

\begin{proof}
For each $1 \leq i \leq M$, define
\begin{equation*}
P_i = \tworow{ccc}{p_{i,1} & \cdots & p_{i,n}}{x_1 & \cdots & x_n}
\end{equation*}
where each $p_{i,j} = p_i$ if $j \in S_i$, and $p_{i,j} = q_i$ otherwise.
Clearly, each $q_i \geq 1$. Because each $0 \leq c_i \leq 1$, $q_i^{-1} \leq 1$, and
\begin{equation*}
\frac{1}{p_i}
= \frac{1}{q_i} + c_i\bigg(1 - \sum_{i=1}^M \frac{1}{q_i}\bigg)
\leq \frac{1}{q_i} + c_i \bigg(1 - \frac{1}{q_i}\bigg)
= c_i \cdot 1 + (1-c_i) \frac{1}{q_i},
\end{equation*}
furthermore $p_i^{-1} \leq 1$,
so each $p_i \geq 1$.
To apply H\"older's inequality, it remains only to prove that $\sum_{i=1}^M P_i^{-1} = 1$
coordinatewise.

For any $j \in \left\{1, \ldots, n\right\}$,
\begin{equation*}
\sum_{i=1}^M \frac{1}{p_{i,j}}
= \sum_{i=1}^M \frac{1}{q_i} + \sum_{S_i \ni j} c_i \epsilon
= 1 - \epsilon + \epsilon \sum_{S_i \ni j} c_i
= 1.
\end{equation*}

Finally, apply H\"older's inequality with mixed norms $P_1, \ldots, P_M$
to functions $f_1, \ldots, f_M$ respectively, followed by a sorting Minkowski.
(Note that each $p_i \leq q_i$, so the $q_i$ norm over the variables
outside of $S_i$ comes first.)
\end{proof}

What follows is perhaps the simplest case of Theorem \ref{thm:new_blei_ps} which gives a new concrete
inequality, not a result of either Proposition \ref{pro:Blei_qp} or Proposition \ref{pro:PopaSinnamon}.
As always, this generalizes to various $\sigma$-finite measure spaces or several distinct functions,
but this result is given in a simple form.

\begin{proposition}
Let $x = x_{i,j,k,l}$ be any quadruply-indexed collection of nonnegative real numbers, where each index
takes at most countably many values. Define
\begin{align*}
A &=
	\Big( \sum_{k,l} \Big( \sum_{i,j} x^{12} \Big)^{1/4}\Big)^{1/3}
	\Big(\sum_{i,j} \Big(\sum_{k,l} x^{12} \Big)^{1/4} \Big)^{1/3},
\\ B &=
	\Big( \sum_{j,l} \Big( \sum_{i,k} x^{12} \Big)^{1/3} \Big)^{1/4}
	\Big( \sum_{i,k} \Big(\sum_{j,l} x^{12} \Big)^{1/3} \Big)^{1/4},
\\ C &=
	\Big( \sum_{j,k} \Big( \sum_{i,l} x^{12} \Big)^{1/2} \Big)^{1/6}
	\Big( \sum_{i,l} \Big(\sum_{j,k} x^{12} \Big)^{1/2} \Big)^{1/6}.
\end{align*}
Then
\begin{equation*}
\sum_{i,j,k,l} x^6 \leq ABC.
\end{equation*}
\end{proposition}

\begin{proof}
For Theorem \ref{thm:new_blei_ps}, let $n=4$ and $k=2$, so that $M = 6$.
Let $q_1 = \cdots = q_6 = 12$, so that $\epsilon = 1 - \sum_{i=1}^6 q_i^{-1} = 1/2$.
Enumerate the two-element subsets of $\left\{1, 2, 3, 4\right\}$ by
$S_1 = \left\{1, 2\right\}, S_2 = \left\{1, 3\right\}, S_3 = \left\{1,4\right\},
S_4 = \left\{2,3\right\}, S_5 = \left\{2,4\right\},$ and $S_6 = \left\{3,4\right\}$.
Observe that
\begin{align*}
1 &\in S_1, S_2, S_3, &
2 &\in S_1, S_4, S_5, \\
3 &\in S_2, S_4, S_6, &
4 &\in S_3, S_5, S_6.
\end{align*}
Let $c_1 = c_6 = 1/2$, $c_2 = c_5 = 1/3$, and $c_3 = c_4 = 1/6$,
so that
\begin{equation*}
c_1 + c_2 + c_3
= c_1 + c_4 + c_5
= c_2 + c_4 + c_6
= c_3 + c_5 + c_6
= 1.
\end{equation*}
Now the result follows from Theorem \ref{thm:new_blei_ps}, noting that
$p_1 = p_6 = 3$, $p_2 = p_5 = 4$, and $p_3 = p_4 = 6$, and letting each function be $x$.
\end{proof}

\textit{Acknowledgements. }
Thanks to Gord Sinnamon for several helpful suggestions.

\bibliographystyle{amsplain}
\bibliography{HM-article}{}


\vspace{4ex}
\parbox{180pt}{
Wayne Grey \\
Department of Mathematics \\
Middlesex College \\
The University of Western Ontario \\
London, Ontario \\
Canada, N6A 5B7 \\
\email{wgrey@uwo.ca}
}


\end{document}